\documentclass{article}
\usepackage{graphicx}
\usepackage{amssymb}
\usepackage{stmaryrd}
\usepackage{enumitem}

\usepackage{geometry}
\usepackage{tikz}
\usepackage{float}
\usepackage{url}
\usepackage{mathtools}
\usepackage{amsmath,amsfonts,amsthm}
\usepackage{mathrsfs} 
\newcommand{\irr}{\operatorname{irr}}

\newtheorem{theorem}{Theorem}[section]
\newtheorem{lemma}[theorem]{Lemma}
\newtheorem{proposition}{Proposition}[section]
\newtheorem{corollary}[theorem]{Corollary}
\newtheorem{definition}{Definition}

\newtheorem{remark}{Remark}

\usepackage{subfig}

\usepackage{ytableau}
\usepackage[colorlinks,
linkcolor=blue,
anchorcolor=blue,
citecolor=blue
]{hyperref}
\begin{document}
\begin{center}
{\large \bf Bounds of Trees with Degree Sequence-Based Topological Indices on Specialized Graph Classes}
\end{center}
\begin{center}
 Jasem Hamoud$^1$ \hspace{0.2 cm}  Duaa Abdullah$^{2}$\\[6pt]
 $^{1,2}$ Physics and Technology School of Applied Mathematics and Informatics \\
Moscow Institute of Physics and Technology, 141701, Moscow region, Russia\\[6pt]
Email: 	 $^{1}${\tt jasem1994hamoud@gmail.com},
 $^{2}${\tt abdulla.d@phystech.edu}
\end{center}
\noindent
\begin{abstract}
In this paper, the investigates Adriatic indices, specifically the sum lordeg index where it defined as $SL(G) = \sum_{u \in V(G)} \deg_G(u) \sqrt{\ln \deg_G(u)}$ and the variable sum exdeg index $SEI_a(G)$ for $a>0$, $a\neq 1$. We present several sharp bounds and characterizations of these and related topological indices on specialized graph classes, including regular graphs, thorny graphs, and chemical trees. Using the strict convexity of function $f$, inequalities for degree-based graph invariants $H_f(T)$ are derived under structural constraints on trees such as branching vertices and maximum degree. Examples on caterpillar trees illustrate the computation of indices like $^{m}M_2(G)$, $F(G)$, $M_2(G)$, and others, revealing the interplay between degree sequences and index values. Additionally, upper and lower bounds on the Sombor index $SO(G^*)$ of thorny graphs $G^*$ are established as 
\[
\operatorname{SO} \leqslant \sum_{uv\in E(G)}\sqrt{\frac{1}{\deg_{G}(u)^2+\deg_{G}(v)^2}+\deg_{G}(u)+\deg_{G}(v)},
\]
 including criteria for equality, with implications for regular and thorn-regular graphs. The treatment includes detailed formulas, constructive examples, and inequalities critical for understanding the relationship between graph topology and vertex-degree-based descriptors.
\end{abstract}

\noindent\textbf{AMS Classification 2010:} 05C05, 05C12, 05C20, 05C25, 05C35, 05C76, 68R10.


\noindent\textbf{UDC:} 519.172.1

\section{Introduction}\label{sec1}

Throughout this paper. Let $G$ be a simple graph, with vertex set $V(G)=\left\{v_{1}, v_{2}, \ldots, v_{n}\right\}$ and edge set $E(G)=\left\{e_{1}, e_{2}, \ldots, e_{m}\right\}$. The quantities $n$ and $m$ are called the order and size of $G$, respectively. The degree of vertex $w$ in $G, d_{G}(w)$, is the number of vertices adjacent to $w$. When from the context it is clear, which graph is considered, the index of $G$ in $d_{G}(w)$ will be omitted. $N(G, v)$ denotes the set of vertices adjacent to $v$ and $N[G, v]=N(G, v) \cup\{v\}$. The maximum degree of $G$, denoted by $\Delta(G)$, and the minimum degree of $G$, denoted by $\delta(G)$, are the maximum and minimum of its vertices' degrees. $\mathscr{D}(G)=\left(d_{G}\left(v_{1}\right), d_{G}\left(v_{2}\right), \ldots, d_{G}\left(v_{n}\right)\right)$, with $d_{G}\left(v_{1}\right) \geqslant d_{G}\left(v_{2}\right) \geqslant \cdots \geqslant d_{G}\left(v_{n}\right)$ is called the degree sequence of $G$. The graph $G$ is said to be regular of degree $k$, when $\mathscr{D}(G)=(k, k, \ldots, k)$. Otherwise, the graph is irregular~\cite{Ghalavand2023AshrafiAD,Cruz2014RadaGutman}. It is obvious that an $n$-vertex tree~\cite{Extremalvertexdegree} has the degree sequence $\left(\operatorname{deg}_{1}, \operatorname{deg}_{2}, \cdots, \operatorname{deg}_{n}\right)$ which is arranged in a non-increasing order if and only if $\sum_{i=1}^{n} \operatorname{deg}_{i}=2(n-1)$. An improper Kragujevac tree~\cite{Cruz2014RadaGutman} is a tree obtained by inserting a new vertex (of degree 2) on a pendent edge of a proper Kragujevac tree. The set of all improper Kragujevac trees of order $ n $ will be denoted by $ K_{g^*}^n $. we use a restricted relation in defining a quasi-regular bipartite graph for the characteristic polynomial. More precisely, $ G $ is a bipartite graph if and only if:
\[
c_{2k}(G) = \Phi_{G}(x) = x^{n} + \left[ \frac{n}{2} \right] x^{n-1} - c_{2k}(G)
\]
where $ c_{2k}(G) $ is a degree in a characteristic index for every $ k = 1, 2, \ldots $.\par
Let $\mathbb{M}_n(\mathbb{C})$ be the set of all square matrices of order $n$ with entries from complex field $\mathbb{C}$. 
For $M \in \mathbb{M}_n(\mathbb{C})$, the square roots of the eigenvalues of $MM^*$ or $M^*M$ are known as the \textit{singular values}~\cite{RatherImran2024Diene},  where $M^*$ is the complex conjugate of $M$. As $MM^*$ is positive semi-definite, so the singular values of $M$ are non-negative real numbers, 
denoted by $\sigma_i(M)$ and are indexed from largest to smallest as: 
$\sigma_1(M) \geq \sigma_2(M) \geq \dots \geq \sigma_n(M)$.  The \textit{Schatten p-norm}~\cite{RatherImran2024Diene} of $M \in \mathbb{M}_n(\mathbb{C})$, denoted by $\|M\|_p$, is the $p$-th root of the sum of the $p$-th powers of $\sigma_i(M)$'s, that is
$\|M\|_p = \left(\sigma_1^p(M) + \sigma_2^p(M) + \dots + \sigma_n^p(M) \right)^{\frac{1}{p}},$
where $1 \leq p \leq n$. \par 
The study analyses graphs using their \emph{spectra} (eigenvalues of matrices associated with the graph)~\cite{RatherImran2024Diene}.   The focus is on the \emph{Szeged matrix} $ S(G) $ and its eigenvalues $\mu_1, \mu_2, \dots, \mu_n$.  The largest eigenvalue $\mu_1$ plays a central role in establishing upper bounds on spectral graph invariants. The unique path~\cite{ZhangXM2013Zhang} connecting two vertices $u,v$ in $T$ will be denoted by $P_{T}(u,v)$.  The  number of edges on $P(u,v)$ is called distance $dist_T(u,v)$, or for short $dist(u,v)$  between them. We call a tree $(T,r)$ rooted at the vertex $r$ (or just by $T$ if it is clear what the root is) by specifying a vertex $r\in V(T)$. The $height$ of a vertex $v$ of a rooted tree $T$ with root $r$ is $h_{T}(v)=dist_{T}(r,v)$. For any two different vertices $u,v$ in a rooted tree $(T,r)$, we say that $v$ is a $successor$ of $u$ and $u$ is an $ancestor$ of $v$ if $P_{T}(r,u)\subset P_{T}(r,v)$.\par  The segment~\cite{Extremalvertexdegree} of a tree $T$ is a path-subtree $S$ whose terminal vertices are branching vertices or pendent vertices of $T$, that is, each internal vertex $u$ of $S$ has $\operatorname{deg}_{T}(u)=2$. The squeeze $S(T)$ of a tree $T$. The edge rotating capacity of vertex $x \in V(T)$ with $2 \leq \operatorname{deg}_{T}(x) \leq \Delta-1$ is defined as $\operatorname{deg}_{T}(x)-1$. The sum of the edge rotating capacities of all vertices with $2 \leq \operatorname{deg}_{T}(x) \leq \Delta-1$ in $T$ is called the total edge rotating capacity of $T$. Let us denote by $PT_{n, p}, S T_{n, s}, BT_{n, b}, DT_{n, k}$ and $MT_{2 m}$ the set of the $n$-vertex trees with $p$.

Let $a_n$  and $b_n$ be two positive sequences. Denote $a_n=\Theta(b_n)$ if $c_1b_n\leq a_n\leq c_2 b_n$ for some positive constants $c_1,c_2$. Denote  $a_n=\omega(b_n)$ if $\lim_{n\rightarrow\infty}\frac{a_n}{b_n}=\infty$.  Let $\mathcal{N}(0,1)$ be the standard normal distribution and $X_n$ be a sequence of random variables~\cite{Yuan2023M}. Then $X_n\Rightarrow\mathcal{N}(0,1)$ means $X_n$ converges in distribution to the standard normal distribution as $n$ goes to infinity. Denote $X_n=O_P(a_n)$ if $\frac{X_n}{a_n}$ is bounded in probability. Denote $X_n=o_P(a_n)$ if $\frac{X_n}{a_n}$ converges to zero in probability as $n$ goes to infinity. Let $\mathbb{E}[X_n]$ and $Var(X_n)$ denote the expectation and variance of a random variable $X_n$ respectively. $\mathbb{P}[E]$ denote the probability of an event $E$. Let $f=f(x,y)$ be a function. For non-negative integers $s,t$, $f^{(s,t)}=f^{(s,t)}(x,y)$ denote the partial derivative $\frac{\partial^{s+t}f(x,y)}{\partial x^s\partial y^t}$. For convenience, sometimes we write $f_x=f^{(1,0)}$, $f_y=f^{(0,1)}$, $f_{xx}=f^{(2,0)}$,$f_{yy}=f^{(0,2)}$ and $f_{xy}=f^{(1,1)}$. $\exp(x)$ denote the exponential function $e^x$.
For positive integer $n$, denote $[n]=\{1,2,\dots,n\}$. Given a finite set $E$, $|E|$ represents the number of elements in $E$.

Our goal of this paper, clarifying the importance of studying topological indices in graph theory due to their impact on graphs, trees, and networks, as well as their influence on other mathematical, physical, and chemical sciences, especially caterpillar trees. Providing further results through trees according to a degree sequence. Also, presenting and explaining many other topological indices.

This paper is organized as follows. In Section~\ref{sec1} we provide the context, motivation, and basic background  and establish the importance of the topic. Based on typical structures and examples from the literature. In Section~\ref{sec2}, we presented sets the foundational concepts, definitions, and notation that the reader needs to understand the rest of the paper. This section introduces the basic terminology and mathematical objects that will be used throughout, ensuring clarity and consistency, through subsection~\ref{subsec2}, we focus on energy of vertices of subdivision graphs. In Section~\ref{sec3} focuses on the study, comparison, and bounds of these and related topological indices, providing a comprehensive theoretical framework motivated by their applications in chemical graph theory and molecular informatics. Through this section we include study of Wiener index~\ref{sub1sec3} and Gutman index~\ref{sub2sec3}. Section~\ref{sec4} had provide random graphs with a given degree sequence $\mathscr{D}=(d_1,\dots,d_n)$. Section~\ref{sec5} focus on discussed some results in many papers had been presented.

 \section{Preliminaries}\label{sec2}
A pendant vertex~\cite{Ghalavand2023AshrafiAD} is a vertex of degree one. For a connected graph $G$ of order $n$ and size $m$, the cyclomatic number of $G$ is defined as $c=m-n+1$.  For a simple graph $G$ of order $n$ and size $m$ with eigenvalues $\lambda_{1} \geqslant \lambda_{2} \geqslant \cdots \geqslant \lambda_{n}$, they proved that $\lambda_{1} \geqslant 2 m / n$ with equality if and only if $G$ is regular~\cite{Ghalavand2023AshrafiAD}.  For $p = 2$, The Schatten 2-norm~\cite{RatherImran2024Diene} is the \textit{Frobenius norm}, denoted by $\|M\|_F$, defined by

$$\|M\|_F^2 = \sigma_1^2(M) + \sigma_2^2(M) + \dots + \sigma_n^2(M).$$

Therefore, an irregularity measure $\operatorname{CS}(G)=\lambda_{1}-2 m / n$. The vertex degree variance of the graph $G, \operatorname{Var}(G)$, as 
\[
\operatorname{Var}(G)=\frac{1}{n} \sum_{v \in V(G)}\left(d(v)-\frac{2 m}{n}\right)^{2}.
\]
To determined the most irregular graphs~\cite{Ghalavand2023AshrafiAD} with respect to irregularity measures $CS$ and $Var$ for various classes of graphs. 
For a nonincreasing sequence of positive integers $\mathscr{D}\footnote{Denote to degree sequence by $\mathscr{D}$.}=(d_0, \cdots,
d_{n-1})$ with $n\ge 3$, let ${\mathcal{T}}_{\mathscr{D}}$ denote the set of
all trees\footnote{Class of trees, meaning it has no cycles and every pair of vertices is connected by exactly one simple path. Trees are a fundamental class of graphs with several important types and related concepts.} with $\mathscr{D}$ as its degree sequence. We can construct a
special tree $T_{\mathscr{D}}^* \in \mathcal{T}_{\mathscr{D}}$ by using
breadth-first search method as follows. Firstly, label the vertex
with the largest degree $d_0$ as $v_{01}$ (the root). Secondly,
label the neighbors of $v_0$ as $v_{11},v_{12},\dots,v_{1d_0}$ from
left to right and let $d(v_{1i})=d_i$ for $i=1, \cdots, d_0$. Then
repeat the second step for all newly labeled vertices until all
degrees are assigned. 
\begin{definition}[\cite{ZhangXM2013Zhang}]
Let $T=(V,E)$ be a tree with root $v_0$. A well-ordering $\prec$ of the vertices is called a BFS-ordering  if $\prec$ satisfies the following properties: 
\begin{itemize}
    \item If $u, v\in V$, and $u\prec v$, then $h(u)\le h(v)$ and $d(u)\ge d(v)$;
    \item If there are two edges $ uu_1\in E(T)$ and
 $vv_1\in E(T)$ such that $u\prec v$, $h(u)=h(u_1)-1$ and $ h(v)=h(v_1)-1$,
 then $u_1\prec v_1$.
\end{itemize}
\end{definition}
Let $\mathscr{D}=(d_0,\cdots, d_{n-1})$ and  $\mathscr{D}^{\prime}=(d_0^{\prime}, \cdots, d_{n-1}^{\prime})$ be two  nonincreasing sequences. If  $\sum_{i=0}^{k}d_i\le\sum_{i=0}^kd_i^{\prime}$ for $k=0, \cdots, n-2$ and  $\sum_{i=0}^{n-1}d_i=\sum_{i=0}^{n-1}d_i^{\prime}$, then the sequence $\mathscr{D}^{\prime}$ is said to {\it major} the sequence $\mathscr{D}$ and denoted by $\mathscr{D}\triangleleft \mathscr{D}^{\prime}$. Assume $T$ be an optimal  tree in $\mathcal{T}_{\mathscr{D}}$ and
$P(x_m,y_m)=x_{m}x_{m-1}\dots x_{2}x_{1}(z)y_{1}y_{2}\dots
y_{m-1}y_m $ be a path of $T$.  If $f_{X_{i}}(x_i)\ge
f_{Y_{i}}(y_{i})$ for $i=1, \cdots, k$  with at least one strict
inequality and $1\le k\le m-1$,   then $f_{X_{\ge k+1}}(x_{k+1})\ge
f_{Y_{\ge k+1}}(y_{k+1})$.

\begin{proposition}[\cite{ZhangXM2013Zhang}]
 Let $\mathscr{D}=(d_0, \cdots d_{n-1})$ and $\mathscr{D}^{\prime}=(d_0^{\prime}, \cdots,
 d_{n-1}^{\prime})$ be two  nonincreasing graphic  degree sequences. If
 $\mathscr{D}\triangleleft \mathscr{D}^{\prime},$ then there exists a series of
 graphic degree sequences  $\mathscr{D}_1, \cdots, \mathscr{D}_k$ such that
 $\mathscr{D} \triangleleft \mathscr{D}_1\triangleleft \cdots \triangleleft
\mathscr{D}_k\triangleleft \mathscr{D}^{\prime}$, where $\mathscr{D}_i$ and $\mathscr{D}_{i+1}$
differ at exactly two entries, say $d_j$ ($d'_j$) and $d_k$ ($d'_k$)
of $\mathscr{D}_i$ ($\mathscr{D}_{i+1}$), with $d'_j = d_j +1$, $d'_k = d_k -1$ and
$j<k$.
\end{proposition}

\begin{theorem}[\cite{YarahmadiZA2023}]
The total irregularity of $G$ is:
\[
\irr_t(G)=\frac{1}{2}\sum_{\{u,v\}\subseteq V(G)}\lvert\deg_u(G)-\deg_v(G)\rvert,
\]
for simple and undirected graph $G$ we have
\[
\irr_t(G)=\begin{cases}
    \frac{1}{12}(2n^3-3n^2-2n+3) & n \text{ is odd}, \\
    \frac{1}{12}(2n^3-3n^2-2n) & n \text{ is even}.
\end{cases}
\]
\end{theorem}
\begin{definition}[\cite{MolloyM1995Reed}]
An asymptotic degree sequence $\mathcal{D}$ is \emph{feasible} if $G_n \neq \emptyset$ for all $n \geq 1$.
\end{definition}
\begin{definition}[\cite{MolloyM1995Reed}]
An \emph{asymptotic degree sequence} $\mathcal{D}$ is \emph{smooth} if there exist constants $\lambda_i$
    such that $\lim_{n \to \infty} d_i(n) = \lambda_i$.
\end{definition}
\begin{definition}[thorn-regular graph~\cite{Gut2025man}]
Let $p_{1}, p_{2}, \ldots, p_{n}$ be non-negative integers. Then the thorny graph of the graph $G$, denoted by $G^{*}$, is obtained by attaching $p_{i}$ pendent vertices to the vertex $v_{i}$, for all $i=1,2, \ldots, n$. Thus, the number of vertices and edges of $G^{*}$ is
$$
n^{*}=n+\sum_{i=1}^{n} p_{i} \quad \text { and } \quad m^{*}=m+\sum_{i=1}^{n} p_{i}
$$
respectively. If $p_{1}=p_{2}=\cdots=p_{n}$, then $G^{*}$ is said to be a thorn-regular graph.
\end{definition}
\begin{theorem}[\cite{Yuan2023M}]
Let $\mathcal{I}_n$ be the topological index of the random graph
$\mathcal{G}_n(\beta, W)$  and $\sigma_n^2$. 
Then
\[
\frac{\mathcal{I}_n-\mathbb{E}[\mathcal{I}_n]}{\sigma_n}\Rightarrow \mathcal{N}(0,1),
\]
as $n$ goes to infinity. 
In addition, the expectation $\mathbb{E}[\mathcal{I}_n]$ has the following asymptotic expression
\[
\mathbb{E}[\mathcal{I}_n]=\left(1+O\left(\frac{1}{np_n}\right)\right)\sum_{1\leq i<j\leq n}p_nw_{ij}f(w_{i(j)},w_{j(i)}),
\]
where the error rate $\frac{1}{np_n}$ cannot be improved.
\end{theorem}
The exponential arithmetic-geometric index~\cite{Das2025BeraHj} (EAG) is a graph invariant defined for a graph $ G $ as:  
$$
EAG(G) = \sum_{v_i v_j \in E(G)} e^{\frac{d_i + d_j}{2 \sqrt{d_i d_j}}},
$$  
where $ d_i, d_j $ are the degrees of adjacent vertices $ v_i, v_j $. This index extends the arithmetic-geometric (AG) index by incorporating an exponential function, emphasizing the ratio of vertex degrees in edge contributions. This index combines the arithmetic and geometric aspects of degrees of adjacent vertices in an exponential form, serving as a topological descriptor in graph theory and mathematical chemistry.  These facilitate comparisons of EAG values under degree constraints and transformations.  The EAG index reflects structural properties through degree correlations, with extremal bounds tied to regular or near-regular graphs.

\subsection{Energy of Vertices of Subdivision Graphs}~\label{subsec2}
The adjacency matrix~\cite{Energyof2025vertices} of a graph $G$ is the square matrix $A = A(G) = [a_{ij}]$ of order $n$, in which $a_{ij} = 1$ if the vertices $v_i$ and $v_j$ are adjacent and $a_{ij} = 0$ otherwise. The characteristic polynomial of a graph $G$ is defined as $\phi(G; \lambda) = \det(\lambda I - A(G)),$
where $I$ is the identity matrix of order $n$. The eigenvalues of $G$ are called the eigenvalues of $G$ and are labeled as $\lambda_1, \lambda_2, \ldots, \lambda_n$.
\begin{definition}[\cite{Energyof2025vertices}]~\label{def1Energyof2025vertices}
The energy of a graph $G$ is defined as the sum of the absolute values of the eigenvalues of $G$ and is denoted by $\mathcal{E}(G)[\lambda]$. That is
\begin{equation}~\label{eq1Energyof2025vertices}
    \mathcal{E}(G) = \sum_{i=1}^{n} |\lambda_i|.
\end{equation}
\end{definition}
Furthermore, according to Definition~\ref{def1Energyof2025vertices}, the energy of a vertex $v_i$, denoted by $\mathcal{E}_G(v_i)$, is defined as $\mathcal{E}_G(v_i) = |A_{ii}|,$ for $i = 1, 2, \ldots, n,$
where $|A_{ii}| = \sqrt{\lambda_i}$. from~\eqref{eq1Energyof2025vertices} by adding the energies of the vertices of $G$. That is
\begin{equation}~\label{eq2Energyof2025vertices}
\mathcal{E}(G) = \sum_{i=1}^{n} \mathcal{E}_G(v_i).
\end{equation}
Consider a graph $G$ with $n$ vertices and adjacency matrix $A$ with eigenvalues $\lambda_1, \ldots, \lambda_n$. The vertex energy at vertex $v_i$ is given by the weighted sum 
\[
\mathcal{E}_G(v_i) = \sum_{j=1}^n p_{ij} |\lambda_j|,
\]
where the weights $p_{ij} = u_{ij}^2$ are squares of components of an orthogonal eigenvector matrix $U$, and satisfy 
\[
\sum_{i=1}^n p_{ij} = \sum_{j=1}^n p_{ij} = 1.
\]
Moreover, the $k$-th moment $\phi_i(A^k)$ with respect to $v_i$ equals the weighted sum 
\[
\phi_i(A^k) = \sum_{j=1}^n p_{ij} \lambda_j^k,
\]
which corresponds to the number of closed walks of length $k$ starting at $v_i$. For bipartite graphs, the total vertex energy across the partite sets is equal, and for $r$-regular graphs, the characteristic polynomial of their subdivision graph relates to the original graph by 
\[
\phi(S(G): \lambda) = \lambda^{m - n} \phi(G: \lambda^2 - r),
\]
where $m$ and $n$ are the numbers of edges and vertices, respectively. Specializing to the complete graph $K_n$ with $n \geq 2$, the subdivision graph $S(K_n)$ has vertex set 
\[
\{v_1, v_2, \ldots, v_n, s_1, s_2, \ldots, s_{\frac{n(n-1)}{2}}\},
\]
where $v_i$ are original vertices and $s_j$ are subdivided vertices. The vertex energy at each original vertex is 
\[
\mathcal{E}_{S(K_n)}(v_i) = \frac{(n-1)\sqrt{n-2} + \sqrt{2(n-1)}}{n},
\]
while at each subdivided vertex it is 
\[
\mathcal{E}_{S(K_n)}(s_j) = \frac{2 ((n-1) \sqrt{n-2} + \sqrt{2(n-1)})}{n (n-1)}.
\]
The trace norm (Sombor energy~\cite{RatherImran2024Diene}) of $S(G)$ is: 
\[
\mathcal{E}(S(G)) = \sum_{i=1}^{n} |\mu_i|.
\]
as the sum of the absolute values of $ \mu_i $ for $ i = 1 $ to $ n $.
\begin{theorem}[\cite{RatherImran2024Diene}]
Let $M$ be a real symmetric matrix of order $n$ and $Q$ be its quotient matrix of order $m$, $(n > m)$. Then the following hold.

\begin{itemize}
    \item[(i)] If the partition $\mathscr{D}$ of $I$ of matrix $M$ is not equitable, then the eigenvalues of $Q$ interlace the eigenvalues of $M$, that is
    $\lambda_i(M) \geq \lambda_i(Q) \geq \lambda_{i+n-m}(M), \quad \text{for } i = 1, 2, \dots, m.$
    \item[(ii)] If the partition $\mathscr{D}$ of $I$ of matrix $M$ is equitable, then the spectrum of $Q$ is contained in the spectrum of $M$.
\end{itemize}
\end{theorem}
\section{Some Novel Degree-Based Topological Indices}~\label{sec3}
The study of topological indices plays an important role in graph theory due to their impact on graphs, trees, and networks, as well as their influence on other mathematical, physical, and chemical sciences, especially caterpillar trees. In paper~\cite{J-D-N2}, we presented the Albertson index as part of our participation in the $66^{th}$ conference held at the Moscow Institute of Physics and Technology. This paper served as a fundamental basis on which we relied to provide further results until paper~\cite{J-D-N3} was prepared, which detailed the Albertson index across trees according to a degree sequence. Through this paper, we review the literature we have worked on during the study period, present some discussed results, and introduce and explain several other topological indices.\par 
In the field of chemical molecular graphs, the atoms are represented by vertices and the bonds by edges that capture the structural essence of compounds. The numerical representation of the molecule graph can be mathematically deduced as a single number, usually called graph invariant or topological index~\cite{LiuWang2019Wang}. The first Zagreb index $M_1(G)$ is the sum of the square of degrees of all the atoms and the second Zagreb index $M_2(G)$ is the sum over all bonds of the product of the vertex degrees of the two adjacent atoms, that is, for any graph $G=(V, E)$ with vertex set $V(G)$ and edge set $E(G)$,
$$
M_1(G)=\sum_{u \in V(G)} d(u)^2, \quad M_2(G)=\sum_{u v \in E(G)} d(u) d(v).
$$
The function associated with the ${ }^{m} M_{2}$ index is reversed to the second Zagreb index~\cite{LiuWang2019Wang,Kumar2024Dasvc,Gut2023man,Yuan2023M}. It is defined by considering the forgotten index $F(G)$ and proposed it as 
\[
{ }^{m} M_{2}(G)=\sum_{u v \in E(G)} \frac{1}{d_{G}(u) d_{G}(v)}, \quad F(G)=\sum_{v \in V(G)} d_{G}(v)^{3}=\sum_{u v \in E(G)}\left(d_{G}(u)^{2}+d_{G}(v)^{2}\right).
\]
The second Zagreb index is a well-known topological index in graph theory, particularly useful in mathematical chemistry. For a simple graph $G$ with vertex degrees $d_i$ and edge set $E(G)$, the second Zagreb index of $G$, often denoted $M_2(G)$, is defined as the sum over all edges of the product of the degrees of the endpoints of each edge. Mathematically, it is given by:
\[
M_2(G) = \sum_{uv \in E(G)} d(u) \cdot d(v),
\]
where $u$ and $v$ are adjacent vertices, and $d(u)$, $d(v)$ denote their degrees. The second Multiplicative Zagreb indices of $k$-tree. The first (generalized) and second Multiplicative Zagreb indices are defined as follows: For $c>0$,
$$
\Pi_{1, c}(G)=\prod_{u \in V(G)} d(u)^c, \quad \Pi_2(G)=\prod_{u v \in E(G)} d(u) d(v)
$$
The second Multiplicative Zagreb index can be rewritten as $\Pi_2(G)=$ $\prod_{u \in V(G)} d(u)^{d(u)}$. Narumi and Katayama considered the product~\cite{Gut2023man} of vertex degrees

$$
N K(G)=\prod_{v} d_{v}(G)
$$
Then, 
\[
\mathscr{D}_{1}(G)=\prod_{v} d_{v}(G)^{2}, \quad \mathscr{D}_{2}(G)=\prod_{u \sim v} d_{u}(G) d_{v}(G), \quad
\mathscr{D}_{1}^{*}(G)=\prod_{u \sim v}\left[d_{u}(G)+d_{v}(G)\right].\]
\begin{lemma}[\cite{LiuWang2019Wang}]
Let $G$ be a connected graph with an edge e and $G^{\prime}$ be obtained from $G$ by deleting the edge e. If $G^{\prime}$ is connected, then
$$
M_1\left(G^{\prime}\right)<M_1(G), \quad M_2\left(G^{\prime}\right)<M_2(G), \quad \Pi_{1, c}\left(G^{\prime}\right)<\Pi_{1, c}(G), \quad \Pi_2\left(G^{\prime}\right)<\Pi_2(G).
$$
\end{lemma}
Let $G \in \mathcal{G}_n$ with $n$ is odd, if $M_1(G)  \leq n(n-1)^2$ and $\Pi_{1, c}(G) \leq(n-1)^{c n}.$ Then
$$ M_2(G) \leq \frac{n(n-1)^3}{2}, \quad \Pi_2(G) 
\leq(n-1)^{n(n-1)}
$$
where the equalities hold if and only if $G \cong K_n$. Assume $G \in \mathcal{G}_n$ with $n$ is even. If $M_1(G)  \leq n(n-2)^2$ and $\Pi_{1, c}(G) \leq(n-2)^{c n}$. Then
$$
M_2(G) \leq \frac{n(n-2)^3}{2}, \quad \Pi_2(G) 
\leq(n-2)^{n(n-2)}
$$
where the equalities hold if and only if $G \cong C P_n$.
\begin{lemma}[\cite{LiuWang2019Wang}]
Let $G$ be a connected graph, possessing (as subgraph) a cycle $Z$ of size $t$. Let $e_1, e_2, \ldots, e_t$ be the edges of $G$ belonging to $Z$. Let $G^{\prime}$ be obtained from $G$ by deleting the edges $e_1, e_2, \ldots, e_t$. If $G^{\prime}$ is connected, then
$M_1\left(G^{\prime}\right)<M_1(G), \quad M_2\left(G^{\prime}\right)<M_2(G), \quad \Pi_{1, c}\left(G^{\prime}\right)<\Pi_{1, c}(G), \quad \Pi_2\left(G^{\prime}\right)<\Pi_2(G)$.
\end{lemma}
 Let $G \in \mathcal{G}_n$ of order $n$. Then, $M_1(G) \geq 4 n,  M_2(G) \geq 4 n, \Pi_{1, c}(G) \geq 2^{c n}, \Pi_2(G) \geq 4^n,$
where the equalities hold if and only if $G \cong C_n$.
The \emph{sum connectivity index} $\operatorname{SCI}$ of a graph $G$ is defined as the sum over all edges of the reciprocal square root of the sum of degrees of the endpoints:
\[
\operatorname{SCI}(G) = \sum_{uv \in E(G)} \frac{1}{\sqrt{d(u) + d(v)}},
\]
where $u$ and $v$ are adjacent vertices, and $d(u)$, $d(v)$ denote their degrees. This index has been studied in chemical graph theory as an indicator of molecular branching and related structural properties. The \emph{symmetric division degree index} $\operatorname{SDD}$ of a simple connected graph $G$ is defined as
\[
\operatorname{SDD}(G) = \sum_{uv \in E(G)} \left(\frac{d(u)}{d(v)} + \frac{d(v)}{d(u)}\right),
\]
where $d(u)$ and $d(v)$ denote the degrees of the vertices $u$ and $v$, respectively. This index captures the symmetry of the division between degrees of adjacent vertices and is used in chemical graph theory and related graph studies to analyze structural properties of molecules and networks. The \emph{inverse sum indeg index} (ISI) of a simple connected graph $G$ is defined as
\[
\operatorname{ISI}(G) = \sum_{uv \in E(G)} \frac{d(u) \cdot d(v)}{d(u) + d(v)},
\]
where $d(u)$ and $d(v)$ denote the degrees of vertices $u$ and $v$, respectively. This index has applications in chemical graph theory, especially as a predictor for molecular properties such as total surface area. The \emph{harmonic index} of a graph $G$, denoted by $H(G)$, is defined as
\[
H(G) = \sum_{uv \in E(G)} \frac{2}{d(u) + d(v)},
\]
where $d(u)$ and $d(v)$ denote the degrees of the vertices $u$ and $v$, respectively. This index has applications in chemistry and graph theory, measuring connectivity in molecular structures and networks. The relationship between the harmonic index and the Wiener index is more subtle since the harmonic index depends on vertex degrees of edges, and the Wiener index depends on distances between all vertex pairs. However, some studies explore bounds and inequalities relating these two indices under certain graph conditions. Typically, graphs with larger average distances (higher Wiener index) tend to have vertices with smaller degrees, which can influence the harmonic index. For example, trees with given orders have known extremal values for both indices, and patterns show connections via graph structure parameters (like degree sequences, diameter, and connectivity). Both are graph invariants used in chemical graph theory and network analysis, with some known bounds connecting them via graph parameters. The \emph{atom-bond connectivity} $\operatorname{ABC}$ index of a graph $G$ is defined as
$$
\operatorname{ABC}(G) = \sum_{uv \in E(G)} \sqrt{\frac{d(u) + d(v) - 2}{d(u) d(v)}},
$$
where $d(u)$ and $d(v)$ denote the degrees of vertices $u$ and $v$, respectively. While the Wiener index incorporates global distance information of the graph, the ABC index involves only local degree information of adjacent vertices. Researchers have studied relationships and bounds involving these indices for certain classes of graphs, and variants of the ABC index have been developed that incorporate distance-based concepts, such as the Graovac-Ghorbani index. The \emph{augmented Zagreb index} $\operatorname{AZI}$ of a graph $G$ is a topological index designed to capture molecular branching and structural information more sensitively than some earlier indices. It is especially important in the study of trees (connected acyclic graphs) because these often model chemical compounds such as alkanes, where the branching structure significantly influences chemical properties, it is defined as
\[
\operatorname{AZI}(G) = \sum_{uv \in E(G)} \left( \frac{d(u) d(v)}{d(u) + d(v) - 2} \right)^3.
\]
 This index is used in chemical graph theory to analyze molecular structure properties and has been found to correlate well with certain physicochemical characteristics. It reacts strongly to changes in vertex degrees, thus effectively distinguishing between different branching patterns in trees, which is critical in chemical graph theory to correlate with molecular properties like stability and reactivity. Studies have shown $\operatorname{AZI}$ correlates well with physicochemical properties in molecular trees such as boiling points, enthalpies, and structural strain, sometimes outperforming other indices. It possesses interesting extremal combinatorial properties, useful for theoretical investigations in graph theory and chemical informatics. The \emph{first hyper-Zagreb index} of a graph $G$, denoted by $HM_1(G)$, is a degree-based topological index widely used in chemical graph theory. It is defined as the sum over all edges of the square of the sum of the degrees of the endpoints of each edge. More formally, if $E(G)$ denotes the edge set of $G$ and $d(u)$ and $d(v)$ are the degrees of the adjacent vertices $u$ and $v$, then

\[
HM_1(G) = \sum_{uv \in E(G)} (d(u) + d(v))^2.
\]

This index generalizes the first Zagreb index by considering the square of the sum of the vertex degrees instead of just their sum. The first hyper-Zagreb index has been found useful in characterizing molecular graphs and correlating structural properties with chemical and physical attributes of the compounds they represent. The \emph{second hyper-Zagreb index} of a graph $G$, denoted by $\mathrm{HM}_2(G)$, is a degree-based graph invariant defined as the sum over all edges of the squares of the products of degrees of the edge endpoints. More precisely, for an edge $uv \in E(G)$ with degrees $d(u)$ and $d(v)$, it is given by
\[
\mathrm{HM}_2(G) = \sum_{uv \in E(G)} \left(d(u) \cdot d(v)\right)^2,
\]
where $d(u)$ and $d(v)$ denote the degrees of vertices $u$ and $v$, respectively. This index is a variant of the classical Zagreb indices and is useful in mathematical chemistry for characterizing molecular structures. The geometric-arithmetic index $GA$, the arithmetic-geometric index $AG$ and defined~\cite{Kumar2024Dasvc} it as

$$
G A(G)=\sum_{u v \in E(G)} \frac{2 \sqrt{d_{G}(u) d_{G}(v)}}{d_{G}(u)+d_{G}(v)}, \quad A G(G)=\sum_{u v \in E(G)} \frac{d_{G}(u)+d_{G}(v)}{2 \sqrt{d_{G}(u) d_{G}(v)}} .
$$
By the arithmetic-geometric mean~\cite{Ghalavand2023AshrafiAD} inequality,
\[
\frac{\left(d_{1}+r\right)+\left(d_{2}+r\right)+\cdots+\left(d_{n}+r\right)}{n} \geqslant \sqrt[n]{\left(d_{1}+r\right)\left(d_{2}+r\right) \cdots\left(d_{n}+r\right)},
\]
with inequality if and only if $d_{1}=d_{2}=\cdots=d_{n}$. Thus,
$$
\frac{n^{n}\left(d_{1}+r\right)\left(d_{2}+r\right) \cdots\left(d_{n}+r\right)}{(2 m+r n)^{n}} \leqslant 1
$$
with equality if and only if $G$ is regular.
\begin{corollary}[\cite{Ghalavand2023AshrafiAD}]
If $T$ is an $n-v e r t e x$ tree. Then,
\begin{align*}
& 1-n^{n}(r+2)^{n-2}(1+r)^{2}(r n+2 n-2)^{-n} \leqslant \mathrm{I}_{A G}(T) \\
\leqslant & 1-n^{n}(1+r)^{n-1}(n-1+r)(r n+2 n-2)^{-n}
\end{align*}
\end{corollary}
\begin{corollary}[\cite{Ghalavand2023AshrafiAD}]~\label{corsomen1}
Let $U$ be a unicyclic graph of order $n \geqslant 4$. Then,
$$
0 \leqslant \mathrm{I}_{A G}(U) \leqslant 1-n^{n}(1+r)^{n-3}(r+2)^{2}(n-1+r)((r+2) n)^{-n}
$$
\end{corollary}
The equality on Corollary~\ref{corsomen1} at the left is satisfied if and only if $U \cong C_{n}$, and the equality on the right is attained if and only if $U \cong S_{n}^{e}$. Denote by $S_{n}^{2 e}$ the graph constructed from $S_{n}$ by adding two edges, so that its degree sequence is $\mathscr{D}\left(S_{n}^{2 e}\right)=(n-1,3,2,2, \overbrace{1, \ldots, 1}^{n-4}) .$ We also, assume that $\mathscr{B}_{n}(2,3)$ is the family of all connected bicyclic graphs with degree sequence $(3,3, \overbrace{2, \ldots, 2}^{n-2}) .$
\begin{corollary}[\cite{Ghalavand2023AshrafiAD}]~\label{corsomen2}
 Let $B$ be a bicyclic $n$- vertex graph. Then, $\mathrm{I}_{A G}(B) \leqslant 1-n^{n}(1+r)^{n-4}(r+2)^{2}(3+r)(n-1+r)(r n+2 n+2)^{-n}, $ 
with equality if and only if $B \cong S_{n}^{2 e}$. Moreover, $\mathrm{I}_{A G}(B) \geqslant 1-n^{n}(r+2)^{n-2}(3+r)^{2}(r n+2 n+2)^{-n},$
with equality if and only if $B \in \mathscr{B}_{n}(2,3)$.
\end{corollary}
According to Corollary~\ref{corsomen2}, assume  $\mathscr{C}_{n}(\Delta-1, \Delta)$ be the family of all connected $c$-cyclic graphs with degree sequence $(\overbrace{\Delta, \ldots, \Delta}^{n_{\Delta}}, \overbrace{\Delta-1, \ldots, \Delta-1}^{n-n_{\Delta}}).$ Among all $c$-cyclic graphs, $c \geqslant 1$, of order $n$ the uniquely determined graph with minimal $I_{A G}$-irregularity is a graph from $\mathscr{C}_{n}(\Delta-1, \Delta)$.

The general Randi\'{c} index $\mathcal{I}_n$ is defined~\cite{Yuan2023M} as
\[
\mathcal{I}_n=\sum_{\{i,j\}\in\mathcal{E}}(d_id_j)^{\tau}. 
\]
When $\tau=-\frac{1}{2}$, $\mathcal{I}_n$ is the Randi\'{c} index.
\begin{corollary}[\cite{Yuan2023M}]
Let $\mathcal{I}_n$ be the general Randi\'{c} index  of the random graph
$\mathcal{G}_n(\alpha,\kappa)$ and $\sigma_n^2$. If $\tau>0$,  
then
\[
\frac{\mathcal{I}_n-\mathbb{E}[\mathcal{I}_n]}{\sigma_n}\Rightarrow \mathcal{N}(0,1),
\]
as $n$ goes to infinity. 
In addition, the expectation $\mathbb{E}[\mathcal{I}_n]$ has the following asymptotic expression
\[
\mathbb{E}[\mathcal{I}_n]=\left(1+O\left(\frac{1}{np_n}\right)\right)\sum_{1\leq i<j\leq n}p_nw_{ij}(w_{i(j)}w_{j(i)})^{\tau},
\]
where the error rate $\frac{1}{np_n}$ cannot be improved.
\end{corollary}

Several corollaries describe bounds of the $AG$-irregularity index $\mathrm{I}_{AG}$ for trees, unicyclic, and bicyclic graphs~\cite{Ghalavand2023AshrafiAD}, identifying extremal graphs with minimal irregularity values based on their degree sequences and cyclic structures. This collection of indices enriches the toolkit for analyzing molecular graphs with both local and global structural perspectives, providing valuable correlations to chemical and physical properties in graph-theoretic and chemical modeling contexts.

\subsection{Wiener Index}~\label{sub1sec3}
The study of topological indices forms a cornerstone in mathematical chemistry, particularly in understanding the structural properties of graphs. Among these indices, the \emph{Wiener index}~\cite{SardarM2025SSD} is a prominent metric, the \emph{Wiener index}~\cite{minimumWienerindex}, proposed by Wiener in 1947, is one of the most studied topological indices in chemical graph theory. Defined for a graph $G$ as the sum of distances between all unordered pairs of vertices. When applied to trees, the Wiener index can be elegantly formulated as:
$$
W(G) = \sum_{u < v} d(u, v), \quad W(G) = \sum_{\varepsilon = ab \in E(G)} n_a n_b
$$
where $d(u, v)$ denotes the distance between vertices $u$ and $v$, and $n_a$ and $n_b$ represent the number of vertices on either side of edge $ab$. For graphs containing cycles, an important relation holds $W(G) \leq S(G)$
where $S(G)$ denotes the Szeged index introduced by Gutman~\cite{SardarM2025SSD}. The Szeged index is given by:
$$
S(G) = \sum_{e} n_a n_b
$$
with the sum taken over all edges $e = ab$ in $G$. This irregularity measure~\cite{Ghalavand2023AshrafiAD} is defined as 
\[
S(G)=\sum_{v \in V(G)}\left|d(v)-\frac{2 m}{n}\right|.
\]
Nikiforov proved that $\operatorname{Var}(G) /(2 \sqrt{2 m}) \leqslant \operatorname{CS}(G) \leqslant$ $\sqrt{S(G)}$  and  $ S(G)^{2} /\left(2 n^{2} \sqrt{2 m}\right) \leqslant \operatorname{CS}(G) \leqslant \sqrt[4]{n^{2} \operatorname{Var}(G)}$ by considering $ S(G)^{2} / n^{2} \leqslant \operatorname{Var}(G) \leqslant S(G)$ and  $\lambda_{n}(G)+\lambda_{n}\left(G^{c}\right) \leqslant-1-S(G)^{2} / 2 n^{3}$ and $\lambda_{k}(G)+\lambda_{n-k+2}\left(G^{c}\right) \leqslant-1-2 \sqrt{2 S(G)}$, where $2 \leqslant k \leqslant n$ and $G^{c}$ denotes the complement of $G$.
Let $G$ be a simple graph of order $n$, size $m$ and with a degree sequence $\mathscr{D}(G)=$ $\left(d_{1}, d_{2}, \ldots, d_{n}\right)$. By the classical result of Euler, $\sum_{i=1}^{n} d_{i}=2 m$. Suppose that $r$ is a non-negative real number~\cite{Ghalavand2023AshrafiAD}.
Additionally, the \emph{first and second Zagreb indices} are defined~\cite{Gut2025man} as:
$$
M_1(G) = \sum_{e} (d_a + d_b), \quad M_2(G) = \sum_{e} d_a d_b
$$
where $d_a$ and $d_b$ are the degrees of the vertices $a$ and $b$, respectively, across all edges $e = ab$. Another key metric is the \emph{irregularity of} $G$, defined as the sum of the absolute differences in degrees over all the edges:
\begin{equation}~\label{eq1}
\irr(G) = \sum_{e}\left|d_a - d_b\right|
\end{equation}

Došić et al. introduced the \emph{Mostar index} $M_{\circ}(G)$, according to~\eqref{eq1} as a novel bond-additive invariant, described as:
$$
M_{\circ}(G) = \sum_{e}\left|n_a - n_b\right|
$$
A Halin graph~\cite{minimumWienerindex} of $G$ is a plane graph consisting of a plane embedding of a tree $T$ of order at least 4 containing no vertex of degree 2, and a cycle $C$ connecting all leaves of $T$. The tree $T$ and the cycle $C$ (or sometimes $C(G)$) are called the characteristic tree and the adjoint cycle of $G$, respectively. Denote by $\mathbb{H}_{n, 4}$ the set of Halin graphs of order $n$ with characteristic trees of diameter 4. In this paper, we determine the minimum Wiener index of graphs in $\mathbb{H}_{n, 4}$ and find the corresponding extremal graphs.  For any $v \in V(G)$, let $N_G(v)$ be the set of neighbors of $v$ and $d_G(v)=\left|N_G(v)\right|$ be the degree of $v$. The distance $d_G(u, v)$ of two vertices $u, v \in V(G)$ is the length of a shortest $u-v$ path in $G$. The greatest distance between any two vertices in $G$ is the diameter of $G$, which is denoted by $\operatorname{diam}(G)$. The eccentricity~\cite{minimumWienerindex} of $v \in V(G)$ is denoted by $\varepsilon(v)$ and is defined as $\varepsilon(v)=\max _{u \in V(G)} d_G(v, u).$ \par 
 Let $G \in \mathbb{H}_{n, 4}$, we use $T(G)$ and $C(G)$ to denote the characteristic tree and the adjoint cycle of $G$. Note that $G$ has only one vertex $o$ of eccentricity 2~\cite{minimumWienerindex}. For a vertex $v \in N_G(o)$, if $v \in C(G)$ then $v$ is called a hanging vertex, otherwise $v$ is called a support vertex~\cite{minimumWienerindex}. We use $[n]$ to denote $\{1,2, \ldots, n\}$. Suppose that $G$ has $m$ support vertices $v_1, v_2, \ldots, v_m$ and $N_G\left(v_i\right) \cap C(G)=\left\{v_i^1, v_i^2, \ldots, v_i^{k_i}\right\}$, where $i \in[m]$ and $k_i \geq 2$. Let $\mathbb{H}_{n, 4}^{\prime}$ be the set of all those graphs of $\mathbb{H}_{n, 4}$ that contain no hanging vertex. Denote by $\mathbb{H}_{n, 4}^*$ the set of all those graphs of $\mathbb{H}_{n, 4}$ that contain at least one hanging vertex.  Some important results and remarks regarding these indices include:

\begin{lemma}[\cite{SardarM2025SSD}]
Let $G$ be a graph with $a, b \in V$ and $f \in E$. Then $d_f = d_a + d_b - 2$.
\end{lemma}

\begin{lemma}[\cite{minimumWienerindex}]
Let $x_n \geq x_{n-1} \geq \cdots \geq x_2 \geq x_1 \geq 2$ be $n$ positive integers with $\sum_{i=1}^n x_i=M$. Define
$$
f\left(x_1, x_2, \ldots, x_n\right)=\sum_{1 \leq i<j \leq n} x_i x_j
$$
where $i, j \in[n]$. Then $f\left(x_1, x_2, \ldots, x_n\right)$ attains the minimum value if and only if there exists $l \in[n]$ such that $x_l=M-2(n-1)$ and $x_i=2$ for $i \in[n] \backslash\{l\}$.
\end{lemma}

\begin{remark}[\cite{SardarM2025SSD}]
Let $T(b_1, b_2, \ldots, b_m)$ (or simply $T_n$) be a star-like tree on $n = lm + 1$ vertices with $b_1 = b_2 = \cdots = b_m = l$. Then $M_0(T_n) = M_0(L_{T_n}) + m l (m - 1).$
\end{remark}

\begin{theorem}[\cite{minimumWienerindex}]
For any graph $G \in \mathbb{H}_{n, 4}^{* 1}$, we have $W(G) \geq F_2(n)$ and
$$
F_2(n)= \begin{cases}106, & \text { if } n=11, \\ 133, & \text { if } n=12 \\ n^2+3 n-46, & \text { if } n \geq 13\end{cases}
$$
where the equality $W(G)=F_2(n)$ holds if and only if $G \cong G_{3,3}$.
\end{theorem}

\begin{theorem}[\cite{SardarM2025SSD}]
Let $T_n$ be a double star tree on $2n + 2$ vertices and $2n + 1$ edges. Then
$M_{\circ}(L_{T_n}) < M_{\circ}(T_n).$
\end{theorem}

\subsection{Gutman Index}~\label{sub2sec3}
The Gutman index~\cite{GeneralGutmanindex} is a topological index used in graph theory, particularly in chemical graph theory. For a connected graph $G$, the Gutman index is defined as the sum over all distinct pairs of vertices $u$ and $v$ of the product of their vertex degrees and the distance between them:

\[
Gut(G) = \sum_{u \neq v} d(u) \cdot d(v) \cdot d(u, v)
\]

where $d(u)$ and $d(v)$ are the degrees of vertices $u$ and $v$, respectively, and $d(u, v)$ is the shortest distance (in terms of edges) between $u$ and $v$ in the graph~\cite{KavithaaSKaladeviV,JayaPercivalMazorodze}. This index combines both the connectivity (through vertex degrees) and the proximity (through distances) information of the graph, making it useful in mathematical chemistry for characterizing molecular structures. We generalize the Gutman index~\cite{GeneralGutmanindex} by introducing the general Gutman index of a connected graph $G$ as
$$
G u t_{a, b}(G)=\sum_{\{u, v\} \subseteq V(G)}\left[d_G(u) d_G(v)\right]^a\left[D_G(u, v)\right]^b,
$$
for $a, b \in \mathbb{R}$. If $a=1$ and $b=1$, we get the classical Gutman index. For $a=0$ and $b=1$, we obtain the Wiener index.

Moreover, there are generalizations such as the general Gutman index defined with parameters altering the powers of degrees and distances, and extensions like the Steiner Gutman index which generalizes the distance notion~\cite{GeneralGutmanindex}.
Actually, through Lemma~\ref{lem1GeneralGutmanindex} we observe the term $Gut_{a,b}\left(G + u_1 u_2\right) < Gut_{a,b}(G)$. For any integers $a$ and $b$ is not zero, through Proposition~\ref{pro1GeneralGutmanindex} we observed that. The Sombor index $SO(G)$, the modified Sombor index ${ }^{m} S O(G)$ are define~\cite{Kumar2024Dasvc,Gut2025man} as:
$$
SO(G)=\sum_{u v \in E(G)} \sqrt{d_{G}(u)^{2}+d_{G}(v)^{2}}, \quad { }^{m} S O(G)=\sum_{u v \in E(G)} \frac{1}{\sqrt{d_{G}(u)^{2}+d_{G}(v)^{2}}}.
$$

\begin{lemma}[\cite{GeneralGutmanindex}]~\label{lem1GeneralGutmanindex}
Let $a \leq 0$ and $b \geq 0$, where at least one of $a$ and $b$ is not zero. Consider a connected graph $G$, and let $u_1, u_2$ be any non-adjacent vertices in $G$. Then we have the inequality $Gut_{a,b}\left(G + u_1 u_2\right) < Gut_{a,b}(G)$.
\end{lemma}

\begin{proposition}[\cite{GeneralGutmanindex}]~\label{pro1GeneralGutmanindex}
Let $G$ be a connected graph with $n$ vertices. For $a \leq 0$ and $b \geq 0$, where at least one of $a$ and $b$ is not zero, the following holds:
\begin{equation}~\label{eq1GeneralGutmanindex}
Gut_{a,b}(G) \geq \frac{n (n-1)^{2a+1}}{2}.
\end{equation}
\end{proposition}
The equality holds if and only if $G$ is the complete graph $K_n$. Furthermore, let $1 \leq x < y$ and $c > 0$. For $a > 1$ or $a < 0$, the inequality $(x+c)^a - x^a < (y+c)^a - y^a$
holds. If $0 < a < 1$, then the inequality is reversed $(x+c)^a - x^a > (y+c)^a - y^a.$ The equality~\eqref{eq2GeneralGutmanindex} holds in Proposition~\ref{pro2GeneralGutmanindex} if and only if $G$ is the complete $k$-partite graph $K_{n_1,n_2,\ldots,n_k}$, where the sizes satisfy $\left|n_i - n_j\right| \leq 1$ for all $1 \leq i < j \leq k$ and $n_1 + n_2 + \cdots + n_k = n$. For the case $a \geq 0$ and $b \leq 0$, where at least one of $a$ and $b$ is not zero, according to~\eqref{eq1GeneralGutmanindex} we have
\[
Gut_{a,b}(G) \leq \frac{n (n-1)^{2a+1}}{2}.
\]
\begin{proposition}[\cite{GeneralGutmanindex}]~\label{pro2GeneralGutmanindex}
Let $a \leq 0$ and $b \geq 0$, where at least one of $a$ and $b$ is not zero. For any $k$-partite graph $G$ with $n$ vertices, where $2 \leq k \leq n$, we have
\begin{equation}~\label{eq2GeneralGutmanindex}
Gut_{a,b}(G) \geq Gut_{a,b}\left(K_{n_1,n_2,\ldots,n_k}\right).
\end{equation}
\end{proposition}
The equality~\eqref{eq3GeneralGutmanindex} holds in Theorem~\ref{thm1GeneralGutmanindex} if and only if $G$ is the complete $k$-partite graph $K_{n_1,n_2,\ldots,n_k}$, by considering Theorem~\ref{thm2GeneralGutmanindex} where the part sizes satisfy $\left|n_i - n_j\right| \leq 1$ for all $1 \leq i < j \leq k$ and $n_1 + n_2 + \cdots + n_k = n$.
\begin{theorem}[\cite{GeneralGutmanindex}]~\label{thm1GeneralGutmanindex}
Let $0 \leq a < \frac{1}{2}$ and $b \leq 0$, where at least one of $a$ and $b$ is not zero. For any $k$-partite graph $G$ with $n$ vertices, where $2 \leq k \leq n$, the reverse inequality holds
\begin{equation}~\label{eq3GeneralGutmanindex}
Gut_{a,b}(G) \leq Gut_{a,b}\left(K_{n_1,n_2,\ldots,n_k}\right).
\end{equation}
\end{theorem}

\begin{theorem}[\cite{GeneralGutmanindex}]~\label{thm2GeneralGutmanindex}
Let $a \leq 0$ and $b \geq 0$, where at least one of $a$ and $b$ is not zero. For any connected graph $G$ with $n$ vertices and chromatic number $\chi$, where $2 \leq \chi \leq n$, we have
\[
Gut_{a,b}(G) \geq Gut_{a,b}\left(K_{n_1,n_2,\ldots,n_\chi}\right).
\]
The equality holds if and only if $G$ is the complete $\chi$-partite graph $K_{n_1,n_2,\ldots,n_\chi}$, where the part sizes satisfy $\left|n_i - n_j\right| \leq 1$ for all $1 \leq i < j \leq \chi$ and $n_1 + n_2 + \cdots + n_\chi = n$.
\end{theorem}

\section{Random Graphs with a Given Degree Sequence}~\label{sec4}
In this section, an \emph{asymptotic degree sequence}~\cite{MolloyM1995Reed} $\mathcal{D} = d_i(n), d_i(n), \ldots$ such that $d_i(n) = 0$ for $i \geq n$ and $\sum_{i=0}^{n} d_i(n) = n$.
Given an asymptotic degree sequence $\mathcal{D}$, we set $\mathcal{D}_n$ to be the degree sequence $\{c_0, c_1, \ldots, c_n\}$, where $c_i = c_i \in \mathbb{N}$ and $\left|\{i : c_i = i\}\right| = d_i(n)$ for each $i \geq 0$. Define $G(n, \mathcal{D}_n)$ to be the set of all graphs with vertex set $[n]$ with degree sequence $\mathcal{D}_n$. A
random graph~\cite{MolloyM1995Reed} on $n$ vertices with degree sequence $\mathcal{D}$ is a uniformly random
member of $G(n, \mathcal{D}_n)$. A graph is said to be random if $A_{ij} (1\leq i<j\leq n)$ are random~\cite{Yuan2023M}. 
\begin{definition}[\cite{Yuan2023M}]
The degree-based topological index of a graph $\mathcal{G}=(\mathcal{V},\mathcal{E})$  is defined as
\[
\mathcal{I}_n=\sum_{\{i,j\}\in \mathcal{E}}f(d_i,d_j),
\]
where $f(x,y)$ is a real function satisfying $f(x,y)=f(y,x)$. 
\end{definition}
\begin{definition}[\cite{Yuan2023M}]
Let $\beta$ be a constant between zero and one, $n$ be an positive integer, and 
\[W=\{w_{ij}\in[\beta,1]| 1\leq i<j\leq n, w_{ji}=w_{ij} ,w_{ii}=0\}.\]
Define a heterogeneous random graph  $\mathcal{G}_n(\beta, W)$ as  
\[\mathbb{P}(A_{ij}=1)=p_nw_{ij},\]
where $A_{ij}$ $(1\leq i<j\leq n)$ are independent, $A_{ij}=A_{ji}$ and $p_n\in(0,1)$.
\end{definition}

Another motivation for studying random graphs on a fixed degree sequence comes from the analysis of the chromatic number of sparse random graphs. This is because a minimally $(r+1)$-chromatic graph must have minimum degree at least $r$. Chvátal's work~\cite{MolloyM1995Reed} indicates that for a random graph  $G_{n,M}$ with a linear number of edges $(c n )$, the expected number of subgraphs with minimum degree 3 is exponentially small if $c < c^* \approx 1.442$, and exponentially large if  $c > c^*$. Further studies suggest that graphs with at least  $1.79n$ edges are minimally 4-chromatic with high probability, hinting that the threshold for 4-chromaticity may depend on more than just subgraphs of minimum degree 3. Luczak's findings~\cite{MolloyM1995Reed} show that a random graph with no vertices of degree less than 2 and at least $\Theta(n)$ vertices of degree greater than 2 almost surely (a.s.) has a unique giant component. This is generalized by the main theorem, which uses a parameter $Q(\mathscr{D}) = \sum_{i=1}^{\infty} d_i - 2\lambda$. If $Q(\mathscr{D}) > 0$, a giant component exists; if $Q(\mathscr{D}) < 0$, all components are small. These results apply to random graph models~\cite{MolloyM1995Reed} like $G_{n,M}$ and $G_{n,p}$ where the degree sequence can be determined accurately and graphs with the same degree sequence are equally likely, allowing verification of known thresholds. The parameter $Q(\mathscr{D})$ reflects the expected change in the number of unknown neighbors during a branching process exploration of a component. A negative $Q(\mathscr{D})$ leads to quick exposure of small components, while a positive $Q(\mathscr{D})$ suggests potential for large components.

\begin{definition}[\cite{MolloyM1995Reed}]
\textit{An asymptotic degree sequence} $\mathfrak{D}$ \textit{is well-behaved if:}

\begin{enumerate}
    \item $\mathfrak{D}$ is feasible and smooth. In this case, $i(i - 2)d_i(n)/n$ tends uniformly to $i(i - 2)\lambda$, i.e., for all $\varepsilon > 0$ there exists $N$ such that for all $n > N$ and for all $i \geq 0$,
    $$\left| \frac{i(i - 2)d_i(n)}{n} - i(i - 2)\lambda \right| < \varepsilon, \quad L(\mathfrak{D}) = \lim_{n \to \infty} \sum_{i \geq 1} i(i - 2) d_i(n)/n.$$
    exists, and the sum approaches the limit uniformly; i.e.,
    \begin{enumerate}
        \item If $L(\mathfrak{D})$ is finite then for all $\varepsilon > 0$ there exists $N$, such that for all $n > N$:
        $$\left| \sum_{i \geq 1} i(i - 2) d_i(n)/n - L(\mathfrak{D}) \right| < \varepsilon.$$
        \item If $L(\mathfrak{D})$ is finite, then, for all $T > 0$, there exists $i^*$, $N$ such that for all $n > N$
        $$\sum_{i > i^*} i(i - 2) d_i(n)/n > T.$$
    \end{enumerate}
\end{enumerate}
\end{definition}
We note that it is an easy exercise to show that if $\mathfrak{D}$ is well-behaved. Then $L(\mathfrak{D}) = \mathcal{Q}(\mathfrak{D}).$
\textit{An asymptotic degree sequence}\cite{MolloyM1995Reed} $\mathfrak{D}$ \textit{is sparse if} 
$$\sum_{i \geq 0} i d_i(n)/n = K + o(1)$$
for some constant $K$.

\begin{lemma}[\cite{MolloyM1995Reed}]
\textit{If a random configuration with a given degree sequence} $\mathfrak{D}$ \textit{meeting the conditions of Theorem 1 [with} $\mathcal{Q}(\mathfrak{D})$ \textit{possibly unbounded] has a property} P \textit{with probability at least} $1 - z^*$ \textit{for some constant} $z < 1$, \textit{then a random graph with the same degree sequence a.s. has} P.
\end{lemma}
\begin{lemma}[\cite{MolloyM1995Reed}]
 If a random configuration with a given degree sequence $\mathfrak{D}$ meeting the conditions of Theorem 1 a.s. has a property P, \textit{and if} $\mathcal{Q}(\mathfrak{D}) < \infty$, then a random graph with the same degree sequence a.s. has P.
\end{lemma}
Given $\mathcal{Q}(\mathfrak{D}) < \infty$, set $\nu = \mathcal{Q}(\mathfrak{D}) / K$ and set $R = 150/\nu^2$. 
Let $F$ be a random configuration~\cite{MolloyM1995Reed} with $n$ vertices and degree sequence $\mathfrak{D}$, meeting the conditions. If $\mathcal{Q}(\mathfrak{D}) < \infty$ and, for some function $0 \leq \omega(n) \leq n^{1/8 - \varepsilon}$, $F$ has no vertices of degree greater than $o(n)$, then $F$ has no components with more than $a = \lfloor R \omega(n) \log n \rfloor$ vertices.
\begin{theorem}[\cite{MolloyM1995Reed}]
If $H$ is a minimally $k$-chromatic graph with minimum degree $k - 1$, then $L(H)$ has no even cycles whose vertices do not induce a clique.
\end{theorem}
\begin{corollary}[\cite{MolloyM1995Reed}]
There exists $\delta > 0$ such that if $G$ is a random graph on $n$ vertices and $\lfloor \delta n \rfloor$ edges, then the expected number of minimally $4$-chromatic subgraphs of $G$ is exponentially small, while the expected number of subgraphs of $G$ with minimum degree at least $3$ is exponentially large in $n$.
\end{corollary}
\begin{corollary}[\cite{RatherImran2024Diene}]
Let $G$ be a connected graph of order $n \geq 2$, $t = 2$ and let $F$ be its forgotten topological index. Then
\begin{equation}~\label{eqq1RatherImran2024Diene}
s(S(G)) \leq \mu_1 + \sqrt{\lVert S(G) \rVert_F^2 - \mu_1^2} = \mu_1 + \sqrt{2F - \mu_1^2} \leq 2\sqrt{F} = \frac{2\lVert S(G) \rVert_F}{\sqrt{2}},
\end{equation}
with equalities holding on both sides if and only if $G$ is the complete bipartite graph.
\end{corollary}
Actually, according to~\eqref{eqq1RatherImran2024Diene} for  a connected graph $G$ of order $n \geq 2$, $t \geq 2$ be an even integer and $k \geq 2$ be the multiplicity of $\mu_n$. Then
$$s(S(G)) \leq \mu_1 + \sqrt{\frac{\lVert S(G) \rVert_F^p - \mu_1^p}{k}},$$
with equality if and only if $G$ is the regular complete $(k+1)$-multipartite graph. The study utilizes the Frobenius norm $\lVert S(G) \rVert_F$, which is defined as:  
  \[
  \lVert S(G) \rVert_F = \sqrt{\sum_{i=1}^{n} \mu_i^2}
  \]
This norm~\cite{RatherImran2024Diene} is used to establish upper bounds on spectral quantities.  
The concept of graph energy, a measure of the sum of absolute eigenvalues, is also referenced indirectly.

\section{Discuss Extremal Vertex-Degree Function Index}~\label{sec5}
The sum lordeg index $SL(G)$ and the variable sum exdeg index $SEI_a(G)$ are two of the Adriatic indices defined as
\[
SL(G) = \sum_{u \in V(G)} \deg_G(u) \sqrt{\ln \deg_G(u)} 
= \sum_{\substack{u \in V(G) \\ \deg_G(u) \geq 2}} \deg_G(u) \sqrt{\ln \deg_G(u)}
\]
and
\[
SEI_a(G) = \sum_{u \in V(G)} \deg_G(u) \, a^{\deg_G(u)}, \quad a > 0, \quad a \neq 1,
\]
respectively~\cite{Extremalvertexdegree}. Notice that $\Pi_1(G)$ and $\Pi_2(G)$ attain their maximum (or minimum) values if and only if
\[
\ln \Pi_1(G) = 2 \sum_{u \in V(G)} \ln \deg_G(u), \quad \ln \Pi_2(G) = \sum_{u \in V(G)} \deg_G(u) \ln \deg_G(u).
\]
are maximum (or minimum), respectively.
\begin{corollary}[\cite{Gut2025man}]
Let $G$ be a regular graph on $n$ vertices, of the degree $r$. Then the Sombor index of a thorn-regular graph of $G$ is
$$
S O\left(G^{*}\right)=\frac{\sqrt{2} n r}{2}(r+p)+n p \sqrt{(r+p)^{2}+1}
$$
\end{corollary}
\begin{corollary}[\cite{Gut2025man}]
If $p_{i}+d\left(v_{i}\right)=D$ holds for all $i=1,2, \ldots, n$, then
$$
S O\left(G^{*}\right)=\sqrt{2} D m+\sqrt{D^{2}+1}\left(n^{*}-n\right)
$$
where $n$ and $m$ are the numbers of the vertices and edges of $G$, whereas $n^{*}$ is the number of the vertices of $G^{*}$.
\end{corollary}
\begin{theorem}[\cite{Extremalvertexdegree}]
Let $T \in P T_{n, p}$ and the function $f(x)$ be strictly convex, where $2 \leq p \leq n-1$. Then

$$
H_{f}(T) \geq[n-(r-1)(n-p)-2] f(r+1)+[(r-1)(n-p)-p+2] f(r)+p f(1) .
$$
\end{theorem}

Equality occurs only if the degree sequence of $T$ is $(\underbrace{r+1, \cdots, r+1}_{n-(r-1)(n-p)-2}, \underbrace{r, \cdots, r}_{(r-1)(n-p)-p+2}, \underbrace{1, \cdots, 1}_{p})$, where $r=\left\lfloor\frac{n-2}{n-p}\right\rfloor+1$.\\
\begin{theorem}[\cite{Extremalvertexdegree}]
 Let $T \in PT_{n, p}$ and $f(x)$ be strictly convex, where $2 \leq p \leq n-1$. Then
$$
H_{f}(T) \leq p f(1)+(n-p-1) f(2)+f(p).
$$
\end{theorem}
Equality occurs only if the degree sequence of $T$ is $(p, \underbrace{2, \cdots, 2}, \underbrace{1, \cdots, 1})$.
\begin{theorem}[\cite{Extremalvertexdegree}]
Let $T \in S T_{n, s}$ and $f(x)$ be strictly convex, where $3 \leq s \leq n-2$. Then
\[
H_{f}(T) \geq \begin{cases}\frac{s-1}{2} f(3)+(n-s-1) f(2)+\frac{s+3}{2} f(1) & \text { ifs is odd, } \\ f(4)+\frac{s-1}{2} f(3)+(n-s-1) f(2)+\frac{s+4}{2} f(1) & \text { ifs is even. }\end{cases}
\]

The equality occurs if and only if the degree sequence of $T$ is $(\underbrace{3, \cdots, 3}, \underbrace{2, \cdots, 2}, \underbrace{1, \cdots, 1})$ for odd $s$ and $(4, \underbrace{3, \cdots, 3}, 2, \cdots, ,\underbrace{1, \cdots, 1)} \text { for even } \mathrm{s}.$,
\end{theorem}
Let $T \in BT_{n, b}$. Then $b \leq \frac{n}{2}-1$ [23]. Since the path is the only tree with no branching vertex, thus, in the following we always assume that $1 \leq b \leq \frac{n}{2}-1$.
\begin{theorem}[\cite{Extremalvertexdegree}]
Let $T \in BT_{n, b}$ and $f(x)$ be strictly convex, where $1 \leq b \leq \frac{n}{2}-1$. Then

$$
H_{f}(T) \geq b f(3)+(n-2 b-2) f(2)+(b+2) f(1)
$$

with the equality holding if and only if the degree sequence of $T$ is $(\underbrace{3, \cdots, 3}_{b}, \underbrace{2, \cdots, 2}_{n-2 b-2}, \underbrace{1, \cdots, 1}_{b+2})$.
\end{theorem}
\begin{lemma}[\cite{Extremalvertexdegree}]
Let $T \in D T_{n, k}$ with the maximum degree $\Delta$. Then $\Delta \leq\left\lfloor\frac{n-2}{k}\right\rfloor+1$.\\
Theorem 6.1. Let $T \in D T_{n, k}$ and $f(x)$ be strictly convex, where $1 \leq k \leq \frac{n}{2}-1$. Then

$$
H_{f}(T) \geq k f(3)+(n-2 k-2) f(2)+(k+2) f(1)
$$

with the equality holding only if the degree sequence of $T$ is $(3, \cdots, 3,2, \cdots, 2,1, \cdots, 1)$.

$$
(\underbrace{3, \cdots, 3}_{k}, \underbrace{2, \cdots, 2}_{n-2 k-2}, \underbrace{1, \cdots, 1}_{k+2}) .
$$
\end{lemma}
Let $\mathscr{F}_{2 m}$ be a collection of $2 m$-vertex trees obtained from $S_{m+1}$ by adding a pendent edge to its $m-1$ pendent vertices. 
\begin{theorem}[\cite{Extremalvertexdegree}]
Let $T \in M T_{2 m}$ and $f(x)$ be strictly convex, where $m \geq 2$. Then

$$
H_{f}(T) \geq 2(m-1) f(2)+2 f(1)
$$

with equality only if $T \cong P_{2 m}$.
\end{theorem}
\begin{theorem}[\cite{Extremalvertexdegree}]
Let $T \in M T_{2 m}$ and $f(x)$ be strictly convex, where $m \geq 2$. Then

$$
H_{f}(T) \leq f(m)+(m-1) f(2)+m f(1)
$$

with equality only if $T \cong \mathscr{F}_{2 m}$.
\end{theorem}
\begin{corollary}[\cite{Extremalvertexdegree}]
Let $T \in B T_{n, b}$ be a chemical tree, where $n \geq 14$ and $1 \leq b \leq \frac{n}{2}-1$. Then

$$
L z(T) \geq 9 b(n-4)+4(n-2 b-2)(n-3)+(b+2)(n-2)
$$

with the equality holding if and only if the degree sequence of $T$ is $(\underbrace{3, \cdots, 3}_{b}, \underbrace{2, \cdots, 2}_{n-2 b-2}, \underbrace{1, \cdots, 1}_{b+2})$.
\end{corollary}
\begin{corollary}[\cite{Extremalvertexdegree}]
Let $T \in D T_{n, k}$ be a chemical tree, where $n \geq 14$ and $1 \leq k \leq \frac{n}{2}-1$. Then

$$
L z(T) \geq 9 k(n-4)+4(n-2 k-2)(n-3)+(k+2)(n-2)
$$

with the equality holding only if the degree sequence of $T$ is $(\underbrace{3, \cdots, 3}_{k}, \underbrace{2, \cdots, 2}_{n-2 k-2}, \underbrace{1, \cdots, 1}_{k+2})$.
\end{corollary}
\begin{corollary}[\cite{Extremalvertexdegree}]
Let $T \in M T_{2 m}$ be a chemical tree, where $m \geq 7$. Then

$$
L z(T) \geq 8(m-1)(2 m-3)+4(m-1)
$$

with equality only if $T \cong P_{2 m}$.
\end{corollary}

For example, if $\mathscr{D}=(4, 4,3,
 3,3,3,3,2,1,1,1, 1, 1, 1,  1, 1, 1,1,1)$, $T^*_{\mathscr{D}}$  is shown in Figure~\ref{fig:placeholder}.
  There is a vertex $v_{01}$ (the root) in layer 0 with the largest degree 4; its four
  neighbors are labeled as $v_{11}, v_{12}, v_{13}, v_{14}$ in layer 1, with
  degrees 4, 3, 3, 3 from left to right; nine
 vertices $v_{21},v_{22}, \cdots, v_{29}$ in layer 2; five
 vertices $v_{31}, v_{32}, v_{33}, v_{34}, v_{35}$ in layer 3. The number
 of vertices in each layer $i$, denoted by $s_i$ can be easily calculated as
 $s_0=1,$
 $ s_1=d_0=4,$ $s_2=d_1+d_2+d_3+d_4-s_1=4+3+3+3-4=9, $ and
 $s_3=d_5+\cdots+d_{13}-s_2=5$.

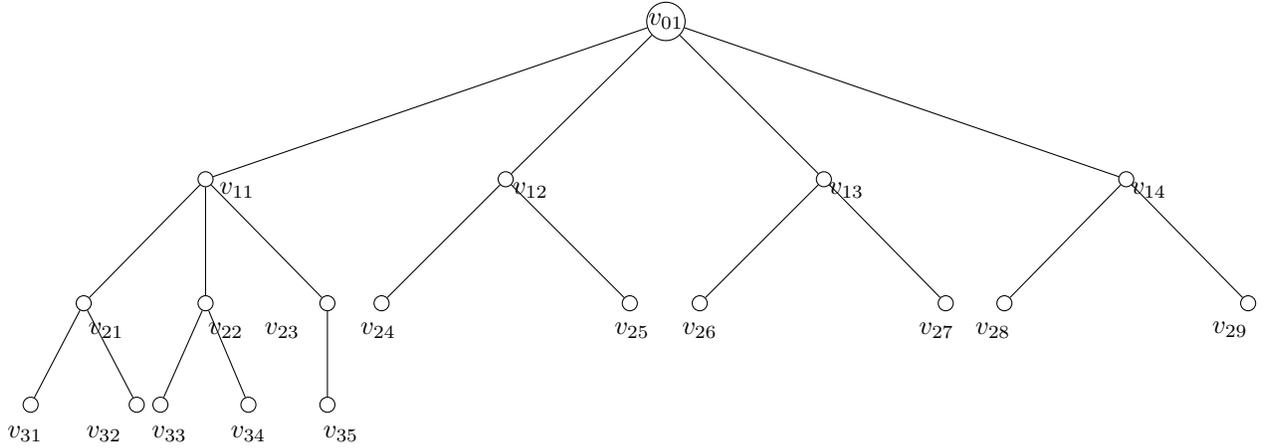
\begin{figure}[H]
    \centering
   \begin{tikzpicture}[scale=.3, every node/.style={draw,circle,minimum size=0.2cm,inner sep=0pt}]
    \node (v01) at (28,27) {$v_{01}$};

    \node (v11) at (7.6,20) {};
    \node[draw=none] at (9,19.5) {$v_{11}$};
    \node (v12) at (20.9,20) {};
    \node[draw=none] at (22,19.5) {$v_{12}$};
    \node (v13) at (35,20) {};
    \node[draw=none] at (36,19.5) {$v_{13}$};
    \node (v14) at (48.4,20) {};
    \node[draw=none] at (49.4,19.5) {$v_{14}$};

    \node (v21) at (2.2,14.5) {};
    \node[draw=none] at (3.2,13.3) {$v_{21}$};
    \node (v22) at (7.6,14.5) {};
    \node[draw=none] at (8.5,13.3) {$v_{22}$};
    \node (v23) at (13,14.5) {};
    \node[draw=none] at (11,13.3) {$v_{23}$};
    \node (v24) at (15.4,14.5) {};
    \node[draw=none] at (15.25,13.3) {$v_{24}$};
    \node (v25) at (26.4,14.5) {};
    \node[draw=none] at (26.5,13.3) {$v_{25}$};
    \node (v26) at (29.5,14.5) {};
    \node[draw=none] at (29.5,13.3) {$v_{26}$};
    \node (v27) at (40.4,14.5) {};
    \node[draw=none] at (40,13.3) {$v_{27}$};
    \node (v28) at (43,14.5) {};
    \node[draw=none] at (42.5,13.3) {$v_{28}$};
    \node (v29) at (53.8,14.5) {};
    \node[draw=none] at (53,13.34) {$v_{29}$};

    \node (v31) at (-0.15,10) {};
    \node[draw=none] at (-0.4,8.7) {$v_{31}$};
    \node (v32) at (4.55,10) {};
    \node[draw=none] at (3.1,8.7) {$v_{32}$};
    \node (v33) at (5.6,10) {};
    \node[draw=none] at (6,8.7) {$v_{33}$};
    \node (v34) at (9.5,10) {};
    \node[draw=none] at (9.5,8.7) {$v_{34}$};
    \node (v35) at (13,10) {};
    \node[draw=none] at (13.6,8.7) {$v_{35}$};

    \draw (v01) -- (v11); \draw (v01) -- (v12); \draw (v01) -- (v13); \draw (v01) -- (v14);

    \draw (v11) -- (v21); \draw (v11) -- (v22); \draw (v11) -- (v23);

    \draw (v12) -- (v24); \draw (v12) -- (v25);

    \draw (v13) -- (v26); \draw (v13) -- (v27);

    \draw (v14) -- (v28); \draw (v14) -- (v29);

    \draw (v21) -- (v31); \draw (v21) -- (v32);
    \draw (v22) -- (v33); \draw (v22) -- (v34);
    \draw (v23) -- (v35);

\end{tikzpicture}

    \caption{Tree with 3 levels.}
    \label{fig:placeholder}
\end{figure}
Assume $\mathscr{D}=(4, 4,3,
 3,3,3,3,2,1,1,1, 1, 1, 1,  1, 1, 1,1,1)$ be a degree sequence in caterpillar tree, through Table~\ref{tabbyjasemn1} we observe the topological indices according to degree sequence. 
 \begin{table}[H]
\centering
\begin{tabular}{|l|l|}
\hline
\text{Index} & \text{Value} \\
\hline
$^{m}M_2(G)$ & 3.424 \\\hline
F(G) & 282 \\\hline
$M_2(G)$ & 103 \\\hline
$NK(G)$ & 7{,}776 \\\hline
$\mathscr{D}_1(G)$ & 60{,}466{,}176 \\\hline
$\mathscr{D}_2(G)$ & 14{,}660{,}155{,}008 \\\hline
$\mathscr{D}_1^*(G)$ & 928{,}972{,}800{,}000 \\\hline
$\operatorname{SCI}(G)$ & 8.515 \\\hline
$\operatorname{SDD}(G)$ & 51.917 \\
\hline
\end{tabular}
\caption{Topological indices with a degree sequence.}
 \label{tabbyjasemn1}
 \end{table}
\begin{theorem}
Consider the degree sequence $\mathscr{D}=(d_1,d_2,\dots,d_n)$ where the edges are $\{d_1,d_2\},\{d_3,d_4\},\dots, \{d_{n-1},d_n\}$. Then, according to Table~\ref{tabjasemv2} we show the Topological indices. 
\end{theorem}
\begin{proof}
Consider a graph $ G = (V,E) $ with vertex set $V = \{1,2,\dots,n\}$
and degree sequence $\mathscr{D} = (d_1, d_2, \dots, d_n).$ The edges are specified as pairs: $E(G) = \{ \{1,2\}, \{3,4\}, \dots, \{n-1, n\} \},$
assuming $ n $ is even for this pairing. We denote the degrees of the vertices in each edge as
\[
a_k = d_{2k-1}, \quad b_k = d_{2k}, \quad k = 1, 2, \dots, \frac{n}{2}.
\]
Thus, the edges are between vertices with degrees $a_k$ and $b_k$.
Substituting $ d(u) = a_k $, $ d(v) = b_k $ for edge $ (2k-1, 2k) $, we have for each index by considering 
\begin{equation}~\label{eq1mainre}
{}^m M_2(G) = \sum_{k=1}^{n/2} \frac{1}{a_k b_k},
\end{equation}
from~\eqref{eq1mainre} we noticed that $F(G) = \sum_{k=1}^{n/2} (a_k^2 + b_k^2)$ which equals 
\begin{equation}~\label{eq2mainre}
F(G)=   \sum_{v=1}^n d_v^3, \quad M_2(G) = \sum_{k=1}^{n/2} a_k b_k.
\end{equation}
Therfore, from~\eqref{eq1mainre} and \eqref{eq2mainre} we established that for 
\begin{equation}~\label{eq3mainre}
NK(G)  = \prod_{k=1}^n d_k = \prod_{k=1}^{n/2} (a_k b_k),
\end{equation}
and 
\[
\mathscr{D}_1(G) = \prod_{k=1}^n d_k^2 = \left( \prod_{k=1}^n d_k \right)^2 = (NK(G))^2 = \left( \prod_{k=1}^{n/2} a_k b_k \right)^2,
\mathscr{D}_2(G) = \prod_{k=1}^{n/2} a_k b_k.
\]
Thus, from~\eqref{eq1mainre}--\eqref{eq3mainre} we find that
\begin{equation}~\label{eq4mainre}
    \mathscr{D}_1^*(G) = \prod_{k=1}^{n/2} (a_k + b_k)
\end{equation}
Finally, from~\eqref{eq4mainre} we have
\[ 
SCI(G) = \sum_{k=1}^{n/2} \frac{1}{\sqrt{a_k + b_k}},
SDD(G) = \sum_{k=1}^{n/2} \left( \frac{a_k}{b_k} + \frac{b_k}{a_k} \right).
\]
Then, 
\begin{table}[H]
\centering
\begin{tabular}{|c|l|c|l|c|l|}
\hline
\textbf{Index} & \textbf{Expression} & \textbf{Index} & \textbf{Expression} & \textbf{Index} & \textbf{Expression} \\
\hline
$^{m}M_2(G)$ & $ \displaystyle \sum_{k=1}^{n/2} \frac{1}{a_k b_k} $ &
$F(G)$ & $ \displaystyle \sum_{k=1}^{n/2} \left( a_k^2 + b_k^2 \right) $ & $M_2(G)$ & $ \displaystyle \sum_{k=1}^{n/2} a_k b_k $ \\\hline
$NK(G)$ & $ \displaystyle \prod_{k=1}^{n/2} a_k b_k $ &
$\mathscr{D}_1(G)$ & $ \displaystyle \prod_{k=1}^n d_k^2 = \left(\prod_{k=1}^{n/2} a_k b_k \right)^2 $ &
$\mathscr{D}_2(G)$ & $ \displaystyle \prod_{k=1}^{n/2} a_k b_k $ \\\hline
$\mathscr{D}_1^{*}(G)$ & $ \displaystyle \prod_{k=1}^{n/2} (a_k + b_k) $ &
$SCI(G)$ & $ \displaystyle \sum_{k=1}^{n/2} \frac{1}{\sqrt{a_k + b_k}} $ &
$SDD(G)$ & $ \displaystyle \sum_{k=1}^{n/2} \left( \frac{a_k}{b_k} + \frac{b_k}{a_k} \right) $ \\
\hline
\end{tabular}
\caption{Topological indices for graph $G$ with edge pairs $\{(2k-1, 2k)\}$ and degrees $a_k = d_{2k-1}, b_k = d_{2k}$.}~\label{tabjasemv2}
\end{table}
As desire.
\end{proof}
Actually, the first geometric index~\cite{J-D-N1} of caterpillar tree with degree sequence $\mathscr{D}=(d_1,\dots,d_n)$ as: 
	\[
	\operatorname{GA}(\mathscr{C})=\frac{4d_1\sqrt{d_1}}{d_1+1} + \frac{2kd_1\sqrt{d_1+1}}{d_1+2} + \frac{4\sqrt{d_1(d_1+1)}}{2d_1+1}.
	\]
where the inequality if and only if $d_2=d_1+1$. Also, for the sequence $d=(d_1,d_2, \underbrace{1,1,\dots}_{n},d_1)$ of $\mathscr{C}(d_1,d_2)$ of order $n\geq 3$, then the sum-connectivity index of caterpillar tree is: 
	\[
	\mathcal{X}(\mathscr{C})=\frac{2(d_1+1)}{\sqrt{d_1+4}} + \frac{kd_1}{\sqrt{d_1+3}} + \frac{2}{\sqrt{2d_1 + 3}}.    \]
	where the inequality if and only if $d_2=d_1+1, d_3=d_2+1$. 
In Table~\ref{tabcomp} we show the value of topological indices by considering the order of vertices is $100$, so that in this case $k=98$.
\begin{table}[H]
\centering
	\begin{tabular}{|c|c|c|c|c|c|c|c|}
		\hline
		$d_1$ & $\operatorname{irr}(\mathscr{C})$ & $\sigma(\mathscr{C})$ & $\mathcal{X}(\mathscr{C})$ & $\operatorname{GA}(\mathscr{C})$ & $\sigma-\operatorname{irr}$&  $\operatorname{GA}-\mathcal{X}$ \\ \hline \hline
		3  & 20    & 55     & 50.289  & 70.942  & 35      & 20.653   \\ \hline
		4  & 34    & 130    & 99.218  & 140.424 & 96      & 41.206   \\ \hline
		5  & 52    & 253    & 171.493 & 238.326 & 201     & 66.833   \\ \hline
		6  & 74    & 430    & 269.944 & 363.615 & 356     & 93.671   \\ \hline
		7  & 100   & 663    & 394.062 & 515.606 & 563     & 121.544  \\ \hline
		8  & 130   & 958    & 543.351 & 693.832 & 828     & 150.481  \\ \hline
		9  & 164   & 1315   & 717.214 & 897.845 & 1151    & 180.631  \\ \hline
		10 & 202   & 1738   & 915.095 & 1127.202& 1536    & 212.107  \\ \hline
		11 & 244   & 2221   & 1136.434& 1381.518& 1977    & 245.084  \\ \hline
		12 & 290   & 2762   & 1380.700& 1660.322& 2472    & 279.622  \\ \hline
		13 & 340   & 3367   & 1647.472& 1963.252& 3027    & 315.780  \\ \hline
		14 & 394   & 4034   & 1936.233& 2290.044& 3640    & 353.811  \\ \hline
		15 & 452   & 4765   & 2246.477& 2630.435& 4313    & 383.958  \\ \hline
		16 & 514   & 5558   & 2577.702& 2994.161& 5044    & 416.459  \\ \hline
		17 & 580   & 6419   & 2929.415& 3380.852& 5839    & 451.437  \\ \hline
		18 & 650   & 7348   & 3301.120& 3790.149& 6698    & 489.029  \\ \hline
		19 & 724   & 8341   & 3692.325& 4221.784& 7617    & 529.459  \\ \hline
		20 & 802   & 9402   & 4102.540& 4675.481& 8600    & 572.941  \\
		\hline
	\end{tabular}
	\caption{Compared Vertices on Caterpillar tree of order $(n,m)$.} \label{tabcomp}
\end{table}
\begin{theorem}[\cite{Gut2025man}]~\label{thm1Gut2025man}
Let $G$ be a simple graph on $n>1$ vertices, and $p_{1}, p_{2}, \ldots, p_{n}$ be non-negative integers. Let $G^{*}$ be the thorny graph of $G$. Then the Sombor
index of $G^{*}$ satisfies the relation
$$
\operatorname{SO}\left(G^{*}\right)=\sum_{e_{i j}} \sqrt{\left[d\left(v_{i}\right)+p_{i}\right]^{2}+\left[d\left(v_{i}\right)+p_{j}\right]^{2}}+\sum_{i=1}^{n} p_{i} \sqrt{\left[d\left(v_{i}\right)+p_{i}\right]^{2}+1} .
$$
\end{theorem}
According to Theorem~\ref{thm1Gut2025man}, we provide the upper bound of Sombor index by next theorem. 

\begin{theorem}~\label{thm2Gut2025man}
Let $G$ be a simple graph of order $n$. Then, the upper bound of Sombor index satisfy
\begin{equation}~\label{eq1thm2Gut2025man}
\operatorname{SO} \leqslant \sum_{uv\in E(G)}\sqrt{\frac{1}{\deg_{G}(u)^2+\deg_{G}(v)^2}+\deg_{G}(u)+\deg_{G}(v)}.
\end{equation}
\end{theorem}
\begin{proof}
Assume $G$ be a simple graph of order $n$, then we noticed that $\deg_{G}(u)+\deg_{G}(v) =\lambda\leqslant \Delta(G)$. Thus, for $\lambda\geqslant 3$, we have $\operatorname{SO} \leqslant \lambda^2(\lambda-1)+2\lambda^3-1.$ Then, the upper bound of Sombor index given as 
\begin{equation}~\label{eq2thm2Gut2025man}
\operatorname{SO}(G) \leqslant \sqrt{\lambda^2(\lambda-1)^3+n\lambda+m(\lambda-1)+\Delta\lambda}.
\end{equation}
Using the inequalities $\frac{1}{\sqrt{2}}(a+b) \leq \sqrt{a^{2}+b^{2}}<a+b$, the following well-known estimates for the Sombor index are straightforwardly obtained:
\begin{equation}~\label{eq03thm2Gut2025man}
\frac{1}{\sqrt{2}} M_{1}(G) \leq S O(G)<M_{1}(G) . \tag{2}
\end{equation}
We noticed that from~\eqref{eq2thm2Gut2025man} holds the relationship $\operatorname{SO}(G) \leq \Delta \, n \, \sqrt{2(n - 1 - \delta)}.$ equality holds if and only if $G$ is regular. Then, 
\begin{equation}~\label{eq3thm2Gut2025man}
\operatorname{SO}(G) \leqslant \sqrt{\frac{1}{\lambda^2} (\lambda-1)^3+n\lambda+2m+\Delta}.
\end{equation}
Therefore, from~\eqref{eq2thm2Gut2025man} and \eqref{eq3thm2Gut2025man} and according to Theorem~\ref{thm1Gut2025man} we noticed that 
\[
\operatorname{SO}\left(G^{*}\right)>\sum_{e_{i j}} \sqrt{\left[d\left(v_{i}\right)+p_{i}\right]^{2}+\left[d\left(v_{i}\right)+p_{j}\right]^{2}},
\]
satisfied according to~\eqref{eq03thm2Gut2025man} an upper bound is sharp for regular graphs, in case $\delta = \Delta = r$, and it simplifies to $\operatorname{SO}(G) = n r \sqrt{2(n-1 - r)}.$ Then, the relationship~\eqref{eq1thm2Gut2025man} holds. As desire.
\end{proof}
\begin{theorem}[\cite{Gut2025man}]
Let $G$ be a simple graph on $n>1$ vertices, and $p_{1}, p_{2}, \ldots, p_{n}$ be non-negative integers. Let $G^{*}$ be the thorny graph of $G$. Then the Sombor index of $G^{*}$ is bounded as
\[
\frac{1}{\sqrt{2}}\left[M_{1}(G)+2 \sum_{i=1}^{n} p_{i} d\left(v_{i}\right)+\sum_{i=1}^{n} p_{i}\left(p_{i}+1\right)\right]<S O\left(G^{*}\right)< M_{1}(G)+2 \sum_{i=1}^{n} p_{i} d\left(v_{i}\right)+\sum_{i=1}^{n} p_{i}\left(p_{i}+1\right).
\]
\end{theorem}
Thus, it is clear that $\operatorname{SO}(G^{*})< M_{1}(G)$.
\begin{theorem}[\cite{Nikiforov2006V}]
If $m\geq n\left(  n-1\right)  /4$. Then, we have $m\sqrt{8m+1}-3m\leq f\left(  n,m\right)  \leq m\sqrt{8m+1}-m.$
Moreover, for $m<\left(  n-1\right)  \left(  n-2\right)  /2,$ then $m\sqrt{8m+1}-m<D\left(  n,m\right),$ where 
\[
f(n,m)  =\max\left\{  \sum_{u\in V\left(  G\right)  }d^{2}\left(
u\right)  :v\left(  G\right)  =n,\text{ }e\left(  G\right)  =m\right\}  .
\]
\end{theorem}

\section{Conclusion}\label{sec6}
Through this paper, the framework of asymptotic degree sequences $\mathcal{D}$ and random graphs with fixed degree sequences $\mathcal{D}_n$. Key parameters such as the threshold function $\mathcal{Q}(\mathfrak{D})$ govern the emergence of giant components in such random graphs, linking combinatorial properties to probabilistic behavior. The definitions and lemmas presented formalize conditions under which random graphs almost surely exhibit structural properties, including chromatic thresholds and component sizes. Additionally, the spectral bounds derived from topological indices and Frobenius norms provide insights into graph energy and eigenvalue constraints. This comprehensive foundation enables further analysis of sparse and well-behaved degree sequences in random graph models. \par 
In conclusion, the explored Adriatic indices, including the sum lordeg index $SL(G)$ and the variable sum exdeg index $SEI_a(G)$, capture essential vertex degree properties of graphs, and their extremal values relate closely to degree sequences and graph classes such as thorn-regular and chemical trees. The derived theorems and corollaries establish sharp bounds and characterize equality cases for several topological indices on various tree families, demonstrating how strict convexity of functions $f$ governs these inequalities. The detailed examples and computations for caterpillar trees provide concrete insights into the structure-index interplay, while bounds on the Sombor index of thorny graphs further illustrate the nuanced dependence of these indices on vertex degrees. Overall, these results deepen the understanding of the interaction between graph structure and vertex-degree-based topological indices.

\section*{Declarations}
\begin{itemize}
\item Funding: Not Funding.
\item Data availability statement: All data is included within the manuscript.
\end{itemize}


\begin{thebibliography}{99}

\bibitem{Cruz2014RadaGutman} Cruz, R., Rada, J., \& Gutman, I. (2014). Topological indices of Kragujevac trees. Proyecciones (Antofagasta), 33(4), 471-482.

\bibitem{Das2025BeraHj}  Das,  K. C., \& Bera,  J. (2025). Resolving an Open Problem on the Exponential Arithmetic-Geometric Index of Unicyclic Graphs. Preprints. \url{https://doi.org/10.20944/preprints202503.1281.v1}.

\bibitem{Gut2023man} Gutman, I. (2013). Degree-based topological indices. Croatica chemica acta, 86(4), 351-361. \url{http://dx.doi.org/10.5562/cca2294}.

\bibitem{Gut2025man} Gutman, I., Sombor index of thorny graphs,2025, Vol. 73, Issue 2, pp.413–422, \url{https://doi.org/10.5937/vojtehg73-56402}.

\bibitem{LiuWang2019Wang}  Liu, JB., Wang, C., Wang, S. et al. Zagreb Indices and Multiplicative Zagreb Indices of Eulerian Graphs. Bull. Malays. Math. Sci. Soc. 42, 67–78 (2019). \url{https://doi.org/10.1007/s40840-017-0463-2}.

\bibitem{SardarM2025SSD} Sardar, M. S., Iqbal, N., Hussain, S., Mukhtar, S., \& Mohammed, W. W. (2025). An inequality for the Mostar index of line graphs of trees. Communications in Combinatorics and Optimization, xx(x), 1–18. \url{https://doi.org/10.22049/cco.2025.30201.2356}.

\bibitem{minimumWienerindex} Wang, H., Xu, K., The minimum Wiener index of Halin graphs with characteristic trees of diameter 4, Electronic Journal of Mathematics, 9 (2025), 11-22, doi: 10.47443/ejm.2024.072.

\bibitem{GeneralGutmanindex} Das, K. C., \& Vetrik, T. (2023). General Gutman index of a graph. MATCH Commun. Math. Comput. Chem, 89, 583-603.	\url{https://doi.org/10.46793/match.89-3.583D}.

\bibitem{JayaPercivalMazorodze} Jaya Percival Mazorodze, Simon Mukwembi, Tomáš Vetrík,
On the Gutman index and minimum degree, Discrete Applied Mathematics,
Volume 173, 2014, Pages 77-82, ISSN 0166-218X, \url{https://doi.org/10.1016/j.dam.2014.04.004}.

\bibitem{KavithaaSKaladeviV} Kavithaa, S., Kaladevi, V., Gutman Index and Detour Gutman Index of Pseudo-Regular Graphs, Journal of Applied Mathematics, 2017, 4180650, 8 pages, 2017. \url{https://doi.org/10.1155/2017/4180650}. 

\bibitem{YarahmadiZA2023} Yarahmadi, Z., A classification of total irregularity of polyomino chains based on segments by using non-decreasing real function, 2023, Int. J. Nonlinear Anal. Appl. In Press, 1–6, \url{http://dx.doi.org/10.22075/ijnaa.2023.29433.4164}.

\bibitem{J-D-N1} Hamoud, J.,\& Abdullah, D. (2025). Regularities of Typical Sigma Index on Caterpillar Trees of Pendent Vertices. arXiv preprint arXiv:2502.10469.

\bibitem{J-D-N3} Hamoud, J., \& Abdullah, D. (2025). Topological Indices With Degree Sequence $\mathscr {D} $ of Tree. arXiv preprint arXiv:2503.12909.
\bibitem{J-D-N2} Hamoud, J., \& Kurnosov, A. (2024). Sigma index in Trees with Given Degree Sequences. arXiv preprint arXiv:2405.05300.

\bibitem{Energyof2025vertices}  H. S. Ramane, S. Y. Chowri, T. Shivaprasad, I. Gutman, Energyof vertices of subdivision graphs, MATCH Commun. Math. Comput.Chem. 93 (2025) 701–711.

\bibitem{Kumar2024Dasvc} Kumar, V., \& Das, S. (2024). On structure sensitivity and chemical applicability of some novel degree-based topological indices. MATCH Communications in Mathematical and in Computer Chemistry, 92(1), 165-203. doi: 10.46793/match.92-1.165K.	

\bibitem{Ghalavand2023AshrafiAD} Ghalavand, A., Ashrafi, A., Dimitrov, D., On The Irregularity Of Graphs Based On The Arithmetic–geometric Mean Inequality, Mathematical Inequalities \& Applications
Volume 26, Number 1 (2023), 151–160, doi:10.7153/mia-2023-26-12.

\bibitem{Nikiforov2006V} Nikiforov, V. (2006). The sum of the squares of degrees: an overdue assignement. arXiv Mathematics e-prints, math-0608660.

\bibitem{MolloyM1995Reed} Molloy, M., \& Reed, B. (1995). A critical point for random graphs with a given degree sequence. Random Structures \& Algorithms, 6(2-3), 161–180, doi:10.1002/rsa.3240060204.

\bibitem{RatherImran2024Diene} Rather, B.A., Imran, M. \& Diene, A. On spectral spread and trace norm of Sombor matrix. Indian J Pure Appl Math (2024). \url{https://doi.org/10.1007/s13226-023-00529-5}.

\bibitem{ZhangXM2013Zhang} Zhang, X. M., Zhang, X. D., Gray, D., \& Wang, H. (2013). The number of subtrees of trees with given degree sequence. Journal of Graph Theory, 73(3), 280-295, doi:10.1002/jgt.21674.

\bibitem{Extremalvertexdegree} Sun, X., Du, J., \& Mei, Y. Extremal vertex-degree function index of trees with some given parameters. Filomat \textbf{39:2} (2025), 659–673,  \url{https://doi.org/10.2298/FIL2502659S}.

\bibitem{Yuan2023M} Yuan, M. (2023). Asymptotic Distribution of Degree-Based Topological Indices. Match - Communications in Mathematical and in Computer Chemistry. \url{https://doi.org/10.46793/match.91-1.135y}.




\end{thebibliography}
\end{document}